\documentclass[11pt]{article}
\usepackage{amsmath,amsthm,amssymb,mathtools}
\usepackage{mathrsfs,graphicx,color,latexsym,tikz,calc}
\usepackage{tikz-cd}
\usepackage{epstopdf}
\usepackage{enumerate}
\usepackage{caption}

\usepackage{xcolor}
\usepackage{hyperref}
\hypersetup{
    colorlinks=true,
    citecolor=blue,
    linkcolor=blue,
    urlcolor=blue
}

\usetikzlibrary{shadows}
\usetikzlibrary{patterns,arrows,decorations.pathreplacing}
\usepackage{ulem}

\newtheorem{theorem}{Theorem}[section]
\newtheorem{lemma}[theorem]{Lemma}
\newtheorem{problem}[theorem]{Problem}

\newtheorem{prop}[theorem]{Proposition}

\theoremstyle{definition}
\newtheorem{definition}[theorem]{Definition}
\newtheorem{remark}[theorem]{Remark}

\newcommand{\Fcal}{\mathcal F}

\newcommand{\Gcal}{\mathcal G}
\newcommand{\Hcal}{\mathcal H}
\newcommand{\Acal}{\mathcal A}
\newcommand{\Kcal}{\mathcal K}
\newcommand{\Lcal}{\mathcal L}
\newcommand{\Mcal}{\mathcal M}
\newcommand{\Scal}{\mathcal S}
\newcommand{\Tcal}{\mathcal T}

\allowdisplaybreaks[4]

\title{\bf \Large
A Spectral Hilton--Milner--Frankl Theorem for $t$-Intersecting Families}

\date{ }

\author{
{\small  Xucheng Bu, \  Lihua Feng, \ Lu Lu\footnote{Corresponding author.\newline{\hspace*{5mm} Email address:} \url{buxcmath@163.com} (X. Bu), \url{fenglh@163.com} (L. Feng), \url{lulumath@csu.edu.cn} (L. Lu), \url{lrr999a@163.com} (R. Lu)}, \ Rongrong Lu}\\[2mm]
\small School of Mathematics and Statistics, HNP-LAMA, Central South University\\
 \small Changsha, Hunan, 410083, China\\
}

\begin{document}
\maketitle

\begin{abstract}
Keevash, Lenz, and Mubayi proved a spectral Erd\H{o}s--Ko--Rado theorem,
showing that, for sufficiently large $n$, the complete $t$-star uniquely
maximizes the adjacency-tensor spectral radius among all $t$-intersecting
$k$-uniform families.

In this paper, we establish a spectral Hilton--Milner--Frankl theorem for
nontrivial $t$-intersecting families in the explicit range
$1\le t\le k-2$ and $n\ge 100\cdot 2^k k^7$. More precisely, we prove that, for every nontrivial $t$-intersecting $k$-uniform family $\mathcal F$, the spectral radius satisfies
\[
 \rho(\mathcal F)\le
 \max\{\rho(\mathcal H_{n,k,t}),\rho(\mathcal A_{n,k,t})\},
\]
where $\mathcal H_{n,k,t}$ and $\mathcal A_{n,k,t}$ are the two extremal
families appearing in the classical Hilton--Milner--Frankl theorem.
Moreover, equality holds only for the extremal candidates attaining the
maximum, up to isomorphism.

We further compare the two candidates asymptotically. For each fixed $t$,
the unique real solution $x=x_t$ of
\[
 (t+2)^{x-t-1}(t+1)^{t+1}=(x-t+1)^{x-1}
\]
determines, as $k$ varies, which of $\mathcal H_{n,k,t}$ and
$\mathcal A_{n,k,t}$ has the larger asymptotic spectral radius.
\end{abstract}



\section{Introduction}
\label{sec:intro}

Let $[n]=\{1,2,\ldots,n\}$. For an integer $k\ge 0$, we write
$\binom{[n]}{k}$ for the family of all $k$-element subsets of $[n]$.
A subfamily $\Fcal\subseteq\binom{[n]}{k}$ is called a
\textit{$k$-uniform hypergraph}, or simply a \textit{$k$-graph}; its members
are called \textit{edges}. A $k$-graph $\Fcal$ is \textit{$t$-intersecting}
if
\[
 |F\cap F'|\ge t \text{ for all } F,F'\in\Fcal.
\]
For $t=1$, we simply say that $\Fcal$ is \textit{intersecting}.

The classical Erd\H{o}s--Ko--Rado theorem determines the maximum size of an
intersecting family of $k$-subsets of $[n]$ in the sharp range $n\ge 2k$.
It states that every intersecting family
$\Fcal\subseteq \binom{[n]}{k}$ satisfies
\[
|\Fcal|\le \binom{n-1}{k-1}.
\]
Moreover, if $n>2k$, equality holds only for full stars.

Wilson extended the Erd\H{o}s--Ko--Rado theorem to $t$-intersecting families
in the large-$n$ range.

\begin{theorem}[Wilson \cite{Wilson}]
\label{thm:wilson}
Suppose that $\Fcal\subseteq\binom{[n]}k$ is $t$-intersecting and
$n>(t+1)(k-t+1)$. Then
\[
 |\Fcal|\le \binom{n-t}{k-t}.
\]
Moreover, equality holds if and only if $\Fcal$ is a
full $k$-uniform $t$-star,
that is, the family of all $k$-sets containing a fixed
$t$-set.
\end{theorem}

The full range of parameters was later determined by the complete
intersection theorem of Ahlswede and Khachatrian
\cite{AhlswedeKhachatrianComplete}.

For any fixed $t$-set $T\subseteq [n]$, the family
\[
\Scal^k_{n,t}(T)=\left\{F\in \binom{[n]}k:T\subseteq F\right\}
\]
is called a \textit{full $k$-uniform $t$-star}. We shall write $\Scal^k_{n,t}$ for an
arbitrary full $k$-uniform $t$-star. A $t$-intersecting family is called \textit{trivial}
if it is contained in some full $k$-uniform $t$-star, and \textit{nontrivial} otherwise.
This leads naturally to the following extremal problem.

\begin{problem}
\label{prob:cardinal-nontrivial}
Determine the maximum size of a nontrivial $t$-intersecting family
$\Fcal\subseteq\binom{[n]}k$, and characterize the extremal families.
\end{problem}

The complete nontrivial-intersection theorem of Ahlswede and Khachatrian
\cite{AhlswedeKhachatrian} gives a full solution to
Problem~\ref{prob:cardinal-nontrivial}: it determines, for every choice of
parameters, the largest size of a nontrivial $t$-intersecting family and
characterizes all extremal families. In the Wilson range
\[
 n>(t+1)(k-t+1),
\]
the extremal families relevant for our purposes are the following two
standard constructions.

\begin{definition}
\label{def:candidate-families}
Define
\begin{align}
 \Hcal_{n,k,t}
 =&
 \left\{H\in\binom{[n]}{k}:[t]\subseteq H,\;
 H\cap[t+1,k+1]\neq\emptyset\right\}
 \cup
 \left\{[k+1]\setminus\{j\}:1\le j\le t\right\},
 \label{eq:H-def}\\
 \Acal_{n,k,t}
 =&
 \left\{A\in\binom{[n]}{k}:|A\cap[t+2]|\ge t+1\right\}.
 \label{eq:A-def}
\end{align}
Both $\Hcal_{n,k,t}$ and $\Acal_{n,k,t}$ are nontrivial
$t$-intersecting families. We refer to $\Hcal_{n,k,t}$ as the Hilton--Milner family and
to $\Acal_{n,k,t}$ as the Frankl family. These two families will be the main extremal
candidates in the spectral problem considered below.
\end{definition}

For $t=1$, Problem~\ref{prob:cardinal-nontrivial} is solved by the classical
Hilton--Milner theorem \cite{HiltonMilner}. For general $t$, important
precursor results were obtained by Frankl \cite{Frankl1978}. With the two
families $\Hcal_{n,k,t}$ and $\Acal_{n,k,t}$ defined above, the relevant
form of the complete nontrivial-intersection theorem of Ahlswede and
Khachatrian \cite{AhlswedeKhachatrian} in the Wilson range can be stated as
follows.

\begin{theorem}[Ahlswede--Khachatrian \cite{AhlswedeKhachatrian}]
\label{thm:AK-Wilson-range}
Suppose that $n>(t+1)(k-t+1)$ and that
$\Fcal\subseteq\binom{[n]}k$ is a nontrivial $t$-intersecting family. Then
\[
|\Fcal|\le
\max\{|\Acal_{n,k,t}|,|\Hcal_{n,k,t}|\}.
\]
Moreover, if equality holds, then $\Fcal$ is isomorphic to one of the
families among $\Acal_{n,k,t}$ and $\Hcal_{n,k,t}$ whose size attains this
maximum.
\end{theorem}

Thus, in the Wilson range, $\Hcal_{n,k,t}$ and $\Acal_{n,k,t}$ are precisely
the cardinal extremal candidates. The purpose of this paper is to prove the
corresponding spectral extremal result.

The fact that two different families, $\Hcal_{n,k,t}$ and $\Acal_{n,k,t}$,
arise as extremal candidates illustrates the additional complexity of the
nontrivial problem. After the full $k$-uniform $t$-stars are excluded, the extremal
structure may depend on the range of the parameters, rather than being given
by a single canonical construction
\cite{AhlswedeKhachatrian,AhlswedeKhachatrianComplete}. Related
Hilton--Milner--Frankl-type problems have been studied in product and
cross-intersection settings; see, for example,
\cite{FranklCross1992,FranklTokushige1992,Borg2008,frankl2024product}.
For further background on intersection problems, we refer the reader to the
monograph of Frankl and Tokushige \cite{FranklTokushige}.

We now turn from cardinal extremal problems to their spectral analogues.
Instead of maximizing the number of edges, a spectral extremal problem asks
for the largest possible value of a spectral parameter associated with the
hypergraph. For a $k$-uniform hypergraph $G\subseteq\binom{[n]}k$, its
adjacency-tensor spectral radius is defined by
\begin{equation}\label{eq:rho-def}
 \rho(G)=k\max\left\{
 \sum_{e\in G}\prod_{i\in e}x_i:
 x_i\ge 0,\ \sum_{i=1}^n x_i^k=1
 \right\}.
\end{equation}

For uniform hypergraphs, the adjacency-tensor framework was introduced by
Cooper and Dutle \cite{CooperDutle}. The variational formulation above is
closely related to the tensor eigenvalue theory of Qi \cite{Qi} and to
Perron--Frobenius theory for nonnegative multilinear forms
\cite{FriedlandGaubertHan}. Keevash, Lenz, and Mubayi introduced the more
general $\alpha$-spectral radius and developed a method for transferring
suitable combinatorial extremal theorems to spectral extremal results
\cite{KLM}. Among other applications, they obtained the following spectral version of
the Erd\H{o}s--Ko--Rado theorem for $t$-intersecting uniform hypergraphs.

\begin{theorem}[Keevash--Lenz--Mubayi \cite{KLM}]
\label{thm:spectral-EKR}
For every $k\ge 2$ and $t\ge 1$, there exists $n_0=n_0(k,t)$ such that, for
all $n\ge n_0$, every $t$-intersecting $k$-uniform family
$\Fcal\subseteq\binom{[n]}k$ satisfies
\[
 \rho(\Fcal)\le \rho(\Scal_{n,t}^k).
\]
Moreover, equality holds if and only if $\Fcal$ is a full $k$-uniform $t$-star.
\end{theorem}

This theorem is a spectral analogue of the Erd\H{o}s--Ko--Rado theorem for
large $n$. Our goal is to establish the corresponding spectral extremal
result for nontrivial $t$-intersecting families.

Several related spectral extensions of intersection theorems for uniform
hypergraphs have also been studied. Zhang and Zhang \cite{ZhangZhang}
obtained sharp results for intersecting uniform hypergraphs with respect to
the adjacency $\mathcal A_\alpha$-tensor and the incidence
$\mathcal Q$-tensor. More broadly, spectral extremal problems for uniform
hypergraphs have recently been investigated from the perspectives of
stability and the $p$-spectral radius; see, for example, Liu, Ni, Wang, and
Kang \cite{LiuNiWangKang}.

We formulate the corresponding nontrivial spectral extremal problem as
follows.

\begin{problem}
\label{prob:spectral-nontrivial}
Determine the largest possible adjacency-tensor spectral radius of a
nontrivial $t$-intersecting $k$-uniform family, and characterize the
families attaining it.
\end{problem}

Recently, Fang, Gao and Chang \cite{FGC} partially solved Problem~\ref{prob:spectral-nontrivial} for the special case $k=3$ and $t=1$ under the assumption that the ground set $[n]$ is large enough. The present paper completely resolves Problem~\ref{prob:spectral-nontrivial} in the
parameter range stated below. We show that, once full $k$-uniform $t$-stars are
excluded, the maximum spectral radius is attained by one of the two explicit
families introduced in Definition~\ref{def:candidate-families}; these are
precisely the two constructions appearing in the Wilson-range form of the
Ahlswede--Khachatrian theorem, Theorem~\ref{thm:AK-Wilson-range}. We also
determine the asymptotic comparison between their spectral radii.

\begin{theorem}
\label{thm:main}
Let $\Fcal\subseteq\binom{[n]}k$ be a nontrivial $t$-intersecting $k$-graph.
Suppose that
\[
1\le t\le k-2 \text{ and }
n\ge100\cdot 2^k k^7.
\]
Then
\begin{equation}
\label{eq:main-bound}
 \rho(\Fcal)\le
 \max\{\rho(\Acal_{n,k,t}),\rho(\Hcal_{n,k,t})\}.
\end{equation}
Moreover, equality holds if and only if $\Fcal$ is isomorphic to a
candidate family whose spectral radius attains the maximum on the
right-hand side. Equivalently,
\begin{itemize}
\item[\rm(i)] if $\rho(\Acal_{n,k,t})\neq\rho(\Hcal_{n,k,t})$, then equality
is attained only by the candidate with the larger spectral radius;
\item[\rm(ii)] if $\rho(\Acal_{n,k,t})=\rho(\Hcal_{n,k,t})$, then equality is
attained by both candidates.
\end{itemize}

Furthermore, for each fixed $t$, the asymptotic comparison of the two
candidates as $n\to\infty$ is as follows. The family $\Acal_{n,k,t}$ has
larger asymptotic spectral radius for
\[
 t+2\le k\le \lfloor\kappa_t\rfloor,
\]
whereas $\Hcal_{n,k,t}$ has larger asymptotic spectral radius for
\[
 k\ge \lceil\kappa_t\rceil.
\]
Here $\kappa_t\notin\mathbb Z$ is the unique real solution of
\begin{equation}
\label{eq:transition-equation}
 (t+2)^{x-t-1}(t+1)^{t+1}=(x-t+1)^{x-1},
 \qquad x>t+1.
\end{equation}
\end{theorem}

Theorem~\ref{thm:main} is proved in Section~\ref{sec:proof-main}, after the
preliminaries in Section~\ref{sec:preliminaries} and the reduction and
decomposition arguments in Section~\ref{sec:reduction_decomposition}. After the proof, we discuss how to determine, for given $n,k,t$, which of
$\Acal_{n,k,t}$ and $\Hcal_{n,k,t}$ has the larger spectral radius.
 The proof relies only
on elementary branching, finite-kernel classification, and tensor
variational estimates; in particular, no asymptotic set-theoretic stability
theorem is required.

\section{Spectral preliminaries}
\label{sec:preliminaries}

In this section we collect the spectral facts used throughout the paper and
fix our normalization. 
Throughout, $\Gcal$ and $\Hcal$ denote $k$-uniform hypergraphs.

For computational convenience, we denote, for $m \ge 1$,
\begin{equation}\label{eq:lambda-m-def}
 \lambda_m^{(k)}(\Gcal)=k!\max\left\{
 \sum_{e\in \Gcal}\prod_{i\in e}x_i:
 x_i\ge0,\ \sum_i x_i^m=1
 \right\},
\end{equation}
and use the following normalization:
\begin{equation}\label{eq:lambda-def}
\lambda(\Gcal)\coloneqq
\lambda_k^{(k)}(\Gcal)
=(k-1)!\rho(\Gcal).
\end{equation}

Thus $\lambda$ differs from the usual spectral radius $\rho$ only by the
constant factor $(k-1)!$. Consequently, all extremal comparisons are unchanged
if $\rho$ is replaced by $\lambda$, and we shall work mostly with
$\lambda$ in the proof.

We shall use the following notation throughout:
\[
 r=t+1,\qquad d=k-t-1,\qquad q=k-t+1,
 \qquad \theta=\frac{k-1}{k}.
\]

We begin with two standard elementary properties.

\begin{lemma}[\cite{nikiforov2014analytic}]
\label{lem:basic}
If $\Gcal\subseteq \Hcal$, then $\lambda_m^{(k)}(\Gcal)\le\lambda_m^{(k)}(\Hcal)$. Moreover,
\[
 \lambda_m^{(k)}(\Gcal\cup \Hcal)\le\lambda_m^{(k)}(\Gcal)+\lambda_m^{(k)}(\Hcal).
\]
\end{lemma}

The following is a strict monotonicity fact of $\lambda$. For a
hypergraph $\Hcal$, its two-section is the graph on the same vertex set in
which two distinct vertices are adjacent whenever they are contained together
in some edge of $\Hcal$.

\begin{lemma}
\label{lem:strict-monotonicity}
Let $\Gcal$ and $\Hcal$ be $k$-uniform hypergraphs on the same vertex set,
with $\Gcal\subsetneq\Hcal$. If the two-section of $\Hcal$ is connected, then
\[
 \lambda(\Gcal)<\lambda(\Hcal).
\]
\end{lemma}
\begin{proof}
Let $\Acal(\Gcal)$ and $\Acal(\Hcal)$ be the adjacency tensors of $\Gcal$
and $\Hcal$. Then $\Gcal\subsetneq\Hcal$ implies
$\Acal(\Gcal)\le \Acal(\Hcal)$ coefficientwise, with strict inequality in at
least one entry. The connectedness of the two-section of $\Hcal$ is
equivalent to weak irreducibility of $\Acal(\Hcal)$. The assertion therefore
follows directly from \cite[Theorem~3.2]{KannanShakedMondererBerman}.
\end{proof}

We shall also use the following edge-count estimate, which is the
$\alpha=k$ case of the bound used by Keevash--Lenz--Mubayi.

\begin{lemma}[\cite{KLM}]
\label{lem:edge-bound}
If a $k$-graph $\Gcal$ has $m$ edges, then
\begin{equation}\label{eq:edge-bound}
 \lambda(\Gcal)\le(k!m)^\theta.
\end{equation}
\end{lemma}

For a vector $x=(x_i)_{i\in[n]}$ and a permutation $\sigma$ of $[n]$, let
$x^\sigma$ denote the coordinate permutation of $x$ defined by
\[
 (x^\sigma)_i=x_{\sigma(i)}.
\]
The next lemma is a group-orbit version of the transposition symmetrization
used in \cite{KLM}. It shows that, when optimizing over a hypergraph with
symmetries, one may choose an optimizer respecting those symmetries.

\begin{lemma}
\label{lem:symmetrization}
Let $m\ge k$, and let $A\le \operatorname{Aut}(\Gcal)$ be a subgroup of the automorphism group
of a $k$-graph $\Gcal$. Then $\Gcal$ has a nonnegative maximizing vector for
\eqref{eq:lambda-m-def} that is constant on every orbit of $A$. In particular,
one may choose a maximizing vector that is constant on every orbit of
$\operatorname{Aut}(\Gcal)$.
\end{lemma}

\begin{proof}
Let $x=(x_i)_{i\in[n]}$ be a nonnegative maximizing vector for
\eqref{eq:lambda-m-def}, and set
\[
 y_i=x_i^m.
\]
Define the average of the $A$-orbit of $y$ by
\[
 \overline y
 =
 \frac{1}{|A|}\sum_{\alpha\in A} y^\alpha,
 \qquad
 \overline x_i=\overline y_i^{1/m}.
\]
We shall show that $\overline x$ is feasible for \eqref{eq:lambda-m-def} and
that replacing $x$ by $\overline x$ does not decrease the objective.

First,
\[
 \sum_{i\in[n]}\overline x_i^m
 =
 \sum_{i\in[n]}\overline y_i
 =
 \frac{1}{|A|}\sum_{\alpha\in A}\sum_{i\in[n]} y_{\alpha(i)}
 =
 \frac{1}{|A|}\sum_{\alpha\in A}\sum_{i\in[n]} y_i
 =
 1.
\]
Thus $\overline x$ is feasible.

Under the change of variables $y_i=x_i^m$, the variational problem in
\eqref{eq:lambda-m-def} may be written as
\[
 \lambda_m^{(k)}(\Gcal)
 =
 k!\max\left\{
 f_{\Gcal}(y): y_i\ge0,\ \sum_{i\in[n]}y_i=1
 \right\},
\]
where
\[
 f_{\Gcal}(y)
 =
 \sum_{e\in\Gcal}
 \left(\prod_{i\in e}y_i\right)^{1/m}.
\]
For every $\alpha\in\operatorname{Aut}(\Gcal)$, we have
\begin{equation}\label{eq:permutatively-invariant}
 f_{\Gcal}(y^\alpha)=\sum_{e\in\Gcal}
 \left(\prod_{i\in e}y_{\alpha(i)}\right)^{1/m}=
 \sum_{e\in\Gcal}
 \left(\prod_{j\in \alpha(e)}y_j\right)^{1/m} =
 \sum_{e'\in \alpha(\Gcal)}
 \left(\prod_{j\in e'}y_j\right)^{1/m}
 =f_{\Gcal}(y),
\end{equation}
since $\alpha(\Gcal)=\Gcal$.

The function $f_{\Gcal}$ is concave on the nonnegative orthant: each summand
\[
 (y_i)_{i\in e}\longmapsto
 \left(\prod_{i\in e}y_i\right)^{1/m}
\]
is a monomial with nonnegative exponents summing to $k/m\le1$, and hence is
concave on the nonnegative orthant.
Therefore Jensen's inequality gives
\[
 f_{\Gcal}(\overline y)
 =
 f_{\Gcal}\left(\frac{1}{|A|}\sum_{\alpha\in A}y^\alpha\right)
 \ge
 \frac{1}{|A|}\sum_{\alpha\in A} f_{\Gcal}(y^\alpha)
 =
 f_{\Gcal}(y),
\]
where the last equality follows from \eqref{eq:permutatively-invariant}.
Thus $\overline x$ is also a maximizing vector.

It remains to check that $\overline x$ is constant on the orbits of $A$.
For every $\beta\in A$,
\[
 \overline y_{\beta(i)}
 =
 \frac{1}{|A|}\sum_{\alpha\in A} y_{\alpha(\beta(i))}
 =
 \frac{1}{|A|}\sum_{\gamma\in A} y_{\gamma(i)}
 =
 \overline y_i,
\]
where we used the change of variables $\gamma=\alpha\beta$. Hence
$\overline x_{\beta(i)}=\overline x_i$ for every $\beta\in A$ and every
$i\in[n]$. Therefore $\overline x$ is constant on each orbit of $A$.
\end{proof}

\section{Kernel reduction and product-layer decomposition}
\label{sec:reduction_decomposition}

This section contains the structural reduction for the main theorem. The first
step is to associate with each maximal nontrivial $t$-intersecting family a
finite kernel, consisting of its minimum $t$-covers. The second step is to
decompose the family into product layers according to its intersections with
this kernel. We then estimate the spectral contribution of these layers.

\subsection{Reduction to the cover kernel}
\label{sec:cover-kernel}

Let $\Gcal\subseteq\binom{[n]}k$ be $t$-intersecting. Its
\textit{$t$-covering number} is
\[
\tau_t(\Gcal)\coloneqq
\min\{|C|: |C\cap G|\ge t \text{ for every } G\in\Gcal\}.
\]
Clearly, $\tau_t(\Gcal)\ge t$. Moreover, $\tau_t(\Gcal)=t$ if and only if
$\Gcal$ is contained in a full $k$-uniform $t$-star. Hence every nontrivial
$t$-intersecting family satisfies $\tau_t(\Gcal)\ge t+1$

We say that $\Gcal\subseteq\binom{[n]}k$ is \textit{maximal} if no
$k$-set outside $\Gcal$ can be added while preserving the
$t$-intersecting property. Equivalently, every $k$-set outside $\Gcal$
fails to $t$-intersect at least one member of $\Gcal$. Every $t$-intersecting
$k$-uniform family can be extended to a maximal one by repeatedly adding a
$k$-set that $t$-intersects every member of the current family. Adding edges
cannot decrease the $t$-covering number, so this extension preserves
nontriviality.

Recall that $r=t+1$. We define
\[
 \Kcal\coloneqq\left\{T\in\tbinom{[n]}r:
 |T\cap G|\ge t\text{ for every }G\in\Gcal\right\},
\]
to be the kernel of $\Gcal$, and let
\[
 \Lcal_n(\Kcal)\coloneqq
 \{F\in\tbinom{[n]}k:T\subseteq F\text{ for some }T\in\Kcal\}
\]
be the lift of $\Kcal$.

\begin{prop}
\label{prop:cover-kernel}
Assume $n\ge 2k+2$, and let
$\Gcal\subseteq\binom{[n]}k$ be a maximal nontrivial $t$-intersecting
family.
\begin{itemize}
\item[\rm (i)] If $\tau_t(\Gcal)\ge t+2$, then $\Gcal$ is covered by at most
$B_0\coloneqq \binom kt k^2$ full $k$-uniform $(t+2)$-stars.

\item[\rm (ii)] If $\tau_t(\Gcal)=t+1$, then
$\Lcal_n(\Kcal)\subseteq\Gcal$, and
$\Gcal\setminus\Lcal_n(\Kcal)$ is covered by at most
$B_R\coloneqq rk^2$ full $k$-uniform $(t+2)$-stars. Moreover, the
kernel $\Kcal$ is a nonempty $t$-intersecting $r$-graph, and one of the
following holds:
\begin{itemize}
\item[\rm (a)] $\Kcal\subseteq\binom Zr$ for some set $Z$ with $|Z|=r+1$;
\item[\rm (b)] $\Kcal\subseteq\{X\cup\{y\}:y\in Y\}$ for some $t$-set $X$
and some $q$-set $Y$ disjoint from $X$.
\end{itemize}
If $\Kcal=\binom Zr$ in {\rm (a)}, then $\Gcal\cong\Acal_{n,k,t}$; if
$\Kcal=\{X\cup\{y\}:y\in Y\}$ in {\rm (b)}, then
$\Gcal\cong\Hcal_{n,k,t}$.
\end{itemize}
\end{prop}

\begin{proof}
Fix $A\in\Gcal$. Suppose first that $\tau_t(\Gcal)\ge t+2$. For each
$F\in\Gcal$, choose a $t$-set $T\subseteq A\cap F$. Since $T$ is not a
$t$-cover, there is $B_T\in\Gcal$ with $|B_T\cap T|\le t-1$. As $F$ and
$B_T$ $t$-intersect and $T\subseteq F$, the set $F$ contains some
$x\in B_T\setminus T$. The set $T\cup\{x\}$ is again not a $t$-cover,
since $\tau_t(\Gcal)\ge t+2$; hence there is $C_{T,x}\in\Gcal$ with
$|C_{T,x}\cap (T\cup\{x\})|\le t-1$. Intersecting $F$ with $C_{T,x}$
forces $F$ to contain some $y\in C_{T,x}\setminus (T\cup\{x\})$. Thus
every $F\in\Gcal$ contains one of the prescribed $(t+2)$-sets
$T\cup\{x,y\}$. There are at most $\binom kt$ choices for $T$, at most
$k$ choices for $x$, and at most $k$ choices for $y$, so $\Gcal$ is
covered by at most $B_0=\binom kt k^2$ full $k$-uniform $(t+2)$-stars.

Now suppose that $\tau_t(\Gcal)=t+1$. Then $\Kcal$ is nonempty. We first
prove that $\Lcal_n(\Kcal)\subseteq\Gcal$. Indeed, every $k$-set
containing a member of $\Kcal$ $t$-intersects every member of $\Gcal$,
and hence belongs to $\Gcal$ by maximality.

Next fix $T\in\Kcal$ and let $F\in\Gcal\setminus\Lcal_n(\Kcal)$. Since
$T$ is a $t$-cover but $T\nsubseteq F$, we have $|F\cap T|=t$. Write
$S=F\cap T=T\setminus\{i\}$ for some $i\in T$. Since $S$ is not a
$t$-cover, choose $B_i\in\Gcal$ with $|B_i\cap S|\le t-1$. As $F$ and
$B_i$ $t$-intersect, $F$ contains some $x\in B_i\setminus S$. The set
$S\cup\{x\}$ is not in $\Kcal$, for otherwise $F$ would contain a kernel
edge; since $|S\cup\{x\}|=r=t+1$, it is not a $t$-cover. Hence there is
$C_{i,x}\in\Gcal$ with $|C_{i,x}\cap(S\cup\{x\})|\le t-1$. Intersecting
$F$ with $C_{i,x}$ forces $F$ to contain some
$y\in C_{i,x}\setminus(S\cup\{x\})$. Thus every member of
$\Gcal\setminus\Lcal_n(\Kcal)$ contains one of at most $rk^2$ prescribed
$(t+2)$-sets, and so this part is covered by at most $B_R=rk^2$ full
$k$-uniform $(t+2)$-stars.

We now show that $\Kcal$ is $t$-intersecting. Otherwise, take
$T_1,T_2\in\Kcal$ with $|T_1\cap T_2|\le t-1$. Since $n\ge 2k+2$, we
may extend them to $k$-sets $G_1\supseteq T_1$ and $G_2\supseteq T_2$
using disjoint new vertices outside $T_1\cup T_2$, so that
$|G_1\cap G_2|=|T_1\cap T_2|$. Each $T_i$ is a $t$-cover, so each $G_i$
$t$-intersects every member of $\Gcal$; by maximality,
$G_1,G_2\in\Gcal$. But then $|G_1\cap G_2|\le t-1$, contradicting that
$\Gcal$ is $t$-intersecting. Hence $\Kcal$ is $t$-intersecting.

It remains to classify $\Kcal$. If $\Kcal$ has a single edge
$T=X\cup\{y\}$ with $|X|=t$, extend $\{y\}$ to a $q$-set
$Y\subseteq[n]\setminus X$. Then
$\Kcal=\{T\}\subseteq\{X\cup\{z\}:z\in Y\}$, so {\rm (b)} holds.
Hence assume that $\Kcal$ contains two distinct edges. Choose distinct
$K_1,K_2\in\Kcal$. Since $|K_1|=|K_2|=r=t+1$ and $\Kcal$ is
$t$-intersecting, $|K_1\cap K_2|=t$. Thus
$K_1=X\cup\{a\}$ and $K_2=X\cup\{b\}$ for some $t$-set $X$ and distinct
$a,b\notin X$.

Suppose first that every edge of $\Kcal$ contains $X$. Then
$\Kcal=\{X\cup\{y\}:y\in Y_0\}$ for some $Y_0\subseteq[n]\setminus X$.
Since $\tau_t(\Gcal)=t+1$, the $t$-set $X$ is not a $t$-cover; choose
$G_0\in\Gcal$ with $|G_0\cap X|\le t-1$. For each $y\in Y_0$, the set
$X\cup\{y\}$ is a $t$-cover, so $|G_0\cap(X\cup\{y\})|\ge t$. Hence
$|G_0\cap X|=t-1$ and $Y_0\subseteq G_0$. Setting $Y=G_0\setminus X$,
we have $|Y|=k-(t-1)=q$, and therefore
$\Kcal\subseteq\{X\cup\{y\}:y\in Y\}$. Thus {\rm (b)} holds.

It remains to consider the case where some $K_3\in\Kcal$ does not
contain $X$. Since $K_3$ meets both $X\cup\{a\}$ and $X\cup\{b\}$ in at
least $t$ points, it must be of the form
$K_3=(X\setminus\{z\})\cup\{a,b\}$ for some $z\in X$. We claim that
every $K_4\in\Kcal$ is contained in $X\cup\{a,b\}$. Indeed, if
$|K_4\cap X|\le t-2$, then $K_4$ cannot meet both $K_1$ and $K_2$ in
$t$ points. If $|K_4\cap X|=t-1$, then meeting both $K_1$ and $K_2$ in
at least $t$ points forces $a,b\in K_4$. Finally, if
$|K_4\cap X|=t$, then $K_4=X\cup\{c\}$ for some $c$, and comparison
with $K_3$ gives $c\in\{a,b\}$. Thus $K_4\subseteq X\cup\{a,b\}$ in all
cases. Taking $Z=X\cup\{a,b\}$ gives
$\Kcal\subseteq\binom Zr$ with $|Z|=t+2=r+1$, so {\rm (a)} holds.

It remains to identify the two equality cases. First suppose that
$\Kcal=\binom Zr$ for some $Z$ with $|Z|=t+2$. Then
$\Lcal_n(\Kcal)$ is, up to relabeling, $\Acal_{n,k,t}$, and
$\Lcal_n(\Kcal)\subseteq\Gcal$. If some
$F\in\Gcal\setminus\Lcal_n(\Kcal)$ existed, then $F$ would contain no
member of $\binom Zr$, so $|F\cap Z|\le t$. If $|F\cap Z|\le t-1$, then
$F$ would meet every $T\in\Kcal$ in fewer than $t$ points, impossible
since $T$ is a $t$-cover. Hence $|F\cap Z|=t$. Choosing
$z\in F\cap Z$ and setting $T=Z\setminus\{z\}\in\Kcal$, we get
$|F\cap T|=t-1$, again impossible. Therefore
$\Gcal=\Lcal_n(\Kcal)\cong\Acal_{n,k,t}$.

Finally suppose that
$\Kcal=\{X\cup\{y\}:y\in Y\}$, where $|X|=t$, $|Y|=q$, and
$X\cap Y=\emptyset$. Since $X$ is not a $t$-cover, choose
$G_0\in\Gcal$ with $|G_0\cap X|\le t-1$. As each $X\cup\{y\}$ is a
$t$-cover, we have $|G_0\cap X|=t-1$ and $Y\subseteq G_0$; since
$(t-1)+q=k$, this gives $G_0=(X\setminus\{x_0\})\cup Y$ for some
$x_0\in X$.

Let $F\in\Gcal\setminus\Lcal_n(\Kcal)$. If $X\subseteq F$, then
$F\cap Y=\emptyset$, and so $|F\cap G_0|\le t-1$, a contradiction. Thus
$X\nsubseteq F$. If $|F\cap X|\le t-2$, then for every $y\in Y$ we have
$|F\cap(X\cup\{y\})|\le t-1$, contradicting that $X\cup\{y\}$ is a
$t$-cover. Therefore $|F\cap X|=t-1$. Since each $X\cup\{y\}$ is a
$t$-cover, this forces $Y\subseteq F$, and hence
$F=(X\setminus\{x\})\cup Y$ for some $x\in X$.

Thus every member of $\Gcal$ outside the lift is one of the $t$
exceptional sets $(X\setminus\{x\})\cup Y$, $x\in X$. These sets
$t$-intersect the lift and pairwise $t$-intersect one another, so by
maximality all of them belong to $\Gcal$. Therefore
\[
 \Gcal
 =
 \Lcal_n(\Kcal)\cup
 \{(X\setminus\{x\})\cup Y:x\in X\}
 \cong \Hcal_{n,k,t},
\]
after relabeling $X$ as $[t]$ and $Y$ as $[t+1,k+1]$.
\end{proof}

\subsection{Product layers and kernel gaps}
\label{sec:product-layer}

Recall that $r=t+1$ and $\theta=1-1/k$. For $s\ge1$, let
\[
 \Scal_s^{(r)}=\{[t]\cup\{y\}:y\in Y\},
\]
where $Y$ is an $s$-set disjoint from $[t]$; this is the $r$-uniform $t$-star with $s$ leaves. Also let
\[
 \Tcal^{(r)}=\binom{[r+1]}r
\]
be the complete $r$-graph on $r+1$ vertices.
Having reduced the possible kernels to subgraphs of the two finite models
$\Scal_q^{(r)}$ and $\Tcal^{(r)}$ in the previous subsection, we now estimate the spectral radii of
the corresponding lifted product layers.

Let $\Kcal$ be an $r$-graph on a fixed vertex set $U\subseteq[n]$ and set
$V=[n]\setminus U$ and $N=|V|$. Define its pure product
layer by
\[
 \Mcal_{n,k}(\Kcal)
 =
 \{T\cup D:T\in\Kcal,\ D\in\tbinom Vd\}
 \subseteq \Lcal_n(\Kcal).
\]
For the $r$-graph $\Kcal$, we use the $r$-uniform $k$-norm normalization of \eqref{eq:lambda-m-def}:
\[
 \lambda_k^{(r)}(\Kcal)
 =
 r!\max\left\{
 \sum_{T\in\Kcal}\prod_{i\in T}y_i:
 y_i\ge0,\ \sum_{i\in U}y_i^k=1
 \right\}.
\]

\begin{lemma}
\label{lem:product}
Let $\Kcal$ be an $r$-graph on a fixed set $U$, let
$V=[n]\setminus U$, $N=|V|$, and set $d=k-r$. Then
\[
 \lambda(\Mcal_{n,k}(\Kcal))
 =
 \frac{k!}{r!}\lambda_k^{(r)}(\Kcal)
 \left(\frac rk\right)^{r/k}
 \left(\frac dk\right)^{d/k}
 \binom Nd N^{-d/k}.
\]
\end{lemma}

\begin{proof}
Let $x\ge0$ with $\sum_i x_i^k=1$, and set
$u=\sum_{i\in U}x_i^k$. The unnormalized objective factors as
\[
 \sum_{F\in\Mcal_{n,k}(\Kcal)}\prod_{i\in F}x_i
 =
 \left(\sum_{T\in\Kcal}\prod_{i\in T}x_i\right)
 \left(\sum_{D\in\binom Vd}\prod_{j\in D}x_j\right).
\]
The first factor is at most
$\lambda_k^{(r)}(\Kcal)u^{r/k}/r!$. For the second, Maclaurin's inequality
together with the power-mean inequality gives
\[
 \sum_{D\in\binom Vd}\prod_{j\in D}x_j
 \le
 \binom Nd N^{-d/k}(1-u)^{d/k}.
\]
Hence the product is at most
\[
 \frac{\lambda_k^{(r)}(\Kcal)}{r!}
 \binom Nd N^{-d/k}
 u^{r/k}(1-u)^{d/k},
\]
which is maximized at $u=r/k$. Multiplying by $k!$ gives the stated upper
bound. Equality is attained by scaling a $k$-norm optimizer for
$\lambda_k^{(r)}(\Kcal)$ to mass $r/k$ on $U$ and taking the constant vector
of mass $d/k$ on $V$.
\end{proof}

Set $\beta_d(N)=d!\binom Nd N^{-d/k}$. Recall that $\theta=1-\frac1k$.
Since $\beta_d(N)=N^{d\theta}
 \prod_{j=0}^{d-1}\left(1-\frac jN\right)$, 
we shall use the simple upper bound
\begin{equation}\label{eq:beta-upper}
 \beta_d(N)\le N^{d\theta}\le n^{d\theta}.
\end{equation}

\begin{lemma}
\label{lem:beta-lower}
If $n\ge2u$ and $n\ge u+d$, then
\begin{equation}\label{eq:beta-lower}
 \beta_d(n-u)\ge n^{d\theta}
 \left(1-\frac{L(u)}n\right)
 \text{ with } L(u)=2d\theta u+d(d-1).
\end{equation}
\end{lemma}

\begin{proof}
The inequality always holds when $L(u)\ge n$. In what follows, we only consider the case for $L(u)< n$. Let $N=n-u$. Then $N\ge n/2$ and $N\ge d$, and hence
\[
 \beta_d(N)
 =
 n^{d\theta}
 \left(1-\frac un\right)^{d\theta}
 \prod_{j=0}^{d-1}\left(1-\frac jN\right).
\]
Since $u/n\le1/2$, we get $
 \left(1-\frac un\right)^{d\theta}
 \ge 1-\frac{2d\theta u}{n}$. Moreover,
\[
 \prod_{j=0}^{d-1}\left(1-\frac jN\right)
 \ge
 1-\sum_{j=0}^{d-1}\frac jN
 =
 1-\frac{d(d-1)}{2N}
 \ge
 1-\frac{d(d-1)}n.
\]
Since $L(u)\le n$, we get $1-\frac{2d\theta u}{n}> 0$ and $1-\frac{d(d-1)}n>0$. Therefore, we get
\[\beta_{d}(N)\ge n^{d\theta}\left(1-\frac{2d\theta u}{n}\right)\left(1-\frac{d(d-1)}n\right)\ge n^{d\theta}\left(1-\frac{L(u)}{n}\right).\]
\end{proof}

\begin{prop}
\label{prop:kernel-radii}
It holds that
\begin{align}
 \lambda_k^{(r)}(\Scal_s^{(r)})
 &=r!r^{-r/k}s^\theta,\label{eq:kernel-star-radius}\\
 \lambda_k^{(r)}(\Tcal^{(r)})
 &=r!(r+1)^{1-r/k}.\label{eq:kernel-top-radius}
\end{align}
Among proper subgraphs of $\Tcal^{(r)}$, the maximum radius is obtained by
deleting one edge and equals
\begin{equation}\label{eq:proper-top}
 r!\,r^{1-(2r-1)/k}(r-1)^{(r-1)/k}.
\end{equation}
\end{prop}
\begin{proof}
We first compute the radius of the $r$-uniform $t$-star
$\Scal_s^{(r)}$. Without loss of generality, we may assume 
\[
 \Scal_s^{(r)}=\{[t]\cup\{j\}:j\in [t+1,t+s]\}.
\]
Let $x=(x_i)$ be a nonnegative vector with $\sum_i x_i^k=1$. Since vertices outside $[t+s]$ do not appear in any edge of
$\Scal_s^{(r)}$, they make no contribution to the objective, so we may
assume that all mass is supported on $[t+s]$.
By Lemma~\ref{lem:symmetrization}, an optimizer may be taken to be constant
on $[t]$ and on $[t+1,t+s]$. Write these two values as $a$ and $b$,
respectively.
Then $ta^k+sb^k=1$ and
\[
 \lambda_k^{(r)}(\Scal_s^{(r)})
 =
 \max r!s a^t b.
\]
Let $u=ta^k$ be the total $k$-mass on $[t]$.
Then $a^k=\frac{u}{t}$, $b^k=\frac{1-u}{s}$,
and the objective equals
\[
 r!s
 \left(\frac{u}{t}\right)^{t/k}
 \left(\frac{1-u}{s}\right)^{1/k}
 =
 r!t^{-t/k}s^\theta
 u^{t/k}(1-u)^{1/k}.
\]
The factor depending on $u$ is
\[
 f(u)=u^{t/k}(1-u)^{1/k},
 \qquad 0\le u\le 1.
\]
Equivalently, we maximize
\[
 f(u)^k=u^{t}(1-u).
\]
Taking logarithms on $(0,1)$, we get
\[
 \log f(u)^k=t\log u+\log(1-u).
\]
Differentiating gives $\frac{d}{du}\log f(u)^k
 =
 \frac{t}{u}-\frac{1}{1-u}$. Thus the unique critical point is determined by $
 \frac{t}{u}=\frac{1}{1-u}$, which gives $t(1-u)=u$,  so $u=\frac{t}{t+1}$. The function is zero at the endpoints and positive inside $(0,1)$, so this
critical point gives the maximum. Substituting $u=t/(t+1)$, we get
\[
\lambda_k^{(r)}(\Scal_s^{(r)})
=
r!t^{-t/k}s^\theta
\left(\frac{t}{r}\right)^{t/k}
\left(\frac1r\right)^{1/k}
=
r!r^{-r/k}s^\theta.
\]

We next compute $\lambda_k^{(r)}(\Tcal^{(r)})$, where $\Tcal^{(r)}=\binom{[r+1]}r$.
Vertices outside $[r+1]$ again make no contribution to the objective, so
we may assume that all mass is supported on $[r+1]$. Since
$\operatorname{Aut}(\Tcal^{(r)})$ acts transitively on $[r+1]$,
Lemma~\ref{lem:symmetrization} gives an optimizer which is constant on
$[r+1]$. Write this common value as $a$.
Then $(r+1)a^k=1$, and hence $a=(r+1)^{-1/k}$.
Since $\Tcal^{(r)}$ has $r+1$ edges, we obtain
\[
 \lambda_k^{(r)}(\Tcal^{(r)})
 =
 r!(r+1)a^r
 =
 r!(r+1)^{1-r/k}.
\]

It remains to determine the largest value of $\lambda_k^{(r)}$ among proper
subgraphs of $\Tcal^{(r)}$. 
Suppose $\Fcal$ misses
at least one edge $E$ of $\Tcal^{(r)}$, and hence $\Fcal\subseteq \Tcal^{(r)}\setminus\{E\}$. By Lemma \ref{lem:basic}, we get $\lambda_k^{(r)}(\Fcal)
 \le
 \lambda_k^{(r)}(\Tcal^{(r)}\setminus\{E\})$.
The one-edge deletion attains the maximum.

All one-edge deletions of $\Tcal^{(r)}$ are isomorphic, so we suppose without loss that the deleted edge is $E=[r]$. Then the edges of
$\Tcal^{(r)}\setminus\{E\}$ are $[r+1]\setminus\{i\}$ for $1\le  i\le r$.
The automorphism group of this hypergraph is transitive on $[r]$ and fixes
the vertex $r+1$. Thus, by Lemma~\ref{lem:symmetrization}, an optimizer
may be taken to have value $a$ on $[r]$ and value $b$ on $\{r+1\}$.
Then $ra^k+b^k=1$,
and
\[
 \lambda_k^{(r)}(\Tcal^{(r)}\setminus\{E\})
 =
 \max r!\,r a^{r-1}b.
\]
Let $u=ra^k$ be the total $k$-mass on the deleted edge $E$. Then
$a^k=\frac{u}{r}$, $b^k=1-u$,
and the objective equals
\[
 r!\,r
 \left(\frac{u}{r}\right)^{(r-1)/k}
 (1-u)^{1/k}.
\]
The factor depending on $u$ is
\[
 u^{(r-1)/k}(1-u)^{1/k}, \qquad 0\le u\le 1.
\]
Replacing $t$ by $r-1$ in the preceding calculation for $f(u)$ shows that this factor
is maximized at $u=\frac{r-1}{(r-1)+1}=\frac{r-1}{r}$.
Substitution gives
\[
\lambda_k^{(r)}(\Tcal^{(r)}\setminus\{E\})
=
r!\,r
\left(\frac{(r-1)/r}{r}\right)^{(r-1)/k}
\left(\frac1r\right)^{1/k}
=
r!\,r^{1-(2r-1)/k}(r-1)^{(r-1)/k}.
\]
\end{proof}

We now record the constants that will be used to compare the two possible
pure product layers.  Recall that, for a kernel $\Kcal$ on $U$, the
leading term of the corresponding pure layer $\Mcal_{n,k}(\Kcal)$ has the form $C(\Kcal)\beta_d(N)$ where $N=n-|U|$. Thus, when $U$ is fixed, the common factor $\beta_d(N)$ is the same
for all kernels on $U$, and the comparison reduces to comparing the coefficients $C(\Kcal)$.

Set
\begin{equation}\label{eq:P}
 P=\frac{k!}{d!}k^{-1}d^{d/k}.
\end{equation}
For the two kernels $\Scal_q^{(r)}$ and $\Tcal^{(r)}$, these coefficients are respectively
\begin{equation}\label{eq:H0A0}
 H_0=Pq^\theta,
 \qquad
 A_0=P\,r^{r/k}(r+1)^{1-r/k}.
\end{equation}

We shall also need the loss in this coefficient when one edge is deleted
from the full kernel.  For the star $\Scal_{q}^{(r)}$ this loss is
\begin{equation}\label{eq:DeltaS}
 \Delta_S=P\bigl(q^\theta-(q-1)^\theta\bigr),
\end{equation}
and for the complete $r$-graph $\Tcal^{(r)}$ it is
\begin{equation}\label{eq:DeltaT}
 \Delta_T=P\left[
 r^{r/k}(r+1)^{1-r/k}
 -r^{1-(r-1)/k}(r-1)^{(r-1)/k}
 \right].
\end{equation}
Both losses are positive.

These are losses in the coefficient of the same factor $\beta_d(N)$.
Therefore $N=n-|U|$, equivalently $|U|$, must be fixed in each
comparison.  Thus, when comparing the full star with $q$ edges and the
star obtained by deleting one edge, we regard both as $r$-graphs on the
same $(k+1)$-element set $U$; after the deletion, one vertex of $U$
is simply unused.  For the complete $r$-graph comparison, the fixed
ground set has size $r+1$.

To estimate the error terms arising from the complete $(t+2)$-stars in the
covering arguments, we record the following exact formula for the spectral
radius of a full $k$-uniform star.
\begin{lemma}
\label{lem:star-formula}
For $0\le c\le k$, we have
\[
 \lambda(\Scal^k_{n,c})
 =
 k!\binom{n-c}{k-c}k^{-1}
 (k-c)^{(k-c)/k}(n-c)^{-(k-c)/k},
\]
where $0^0$ is interpreted as $1$.
\end{lemma}

\begin{proof}
Fix the center $C$ of the star, with $|C|=c$. By
Lemma~\ref{lem:symmetrization}, an optimizer may be taken to be constant
on $C$ and on $[n]\setminus C$. Write these two values as $a$ and $b$.
Then $c a^k+(n-c)b^k=1$ and
\[
 \lambda(\Scal^k_{n,c})
 =
 \max k!\binom{n-c}{k-c}a^c b^{k-c}.
\]
Assume first that $0<c<k$. Let $u=c a^k$ be the total $k$-mass on the
center. Then $(n-c)b^k=1-u$, and the objective equals
\[
 k!\binom{n-c}{k-c}
 \left(\frac{u}{c}\right)^{c/k}
 \left(\frac{1-u}{n-c}\right)^{(k-c)/k}.
\]
The factor depending on $u$ is
$u^{c/k}(1-u)^{(k-c)/k}$, which is maximized on $[0,1]$ at
$u=c/k$. Substitution gives the claimed formula.

The endpoint cases are consistent with the same expression: when $c=0$,
the star is the complete $k$-graph on $n$ vertices and the value is
$k!\binom nk/n$; when $c=k$, the star consists of one edge and the value is
$k!/k=(k-1)!$. These are exactly the displayed formula under the convention
$0^0=1$.
\end{proof}

Recall that $d=k-t-1$. Lemma~\ref{lem:star-formula} with $c=r+1=t+2$ yields
\begin{equation}\label{eq:Z-bound}
 \lambda(\Scal^k_{n,r+1})
 \le Z\,n^{(d-1)\theta}
 \text{ with }
 Z\coloneqq \frac{k!}{(d-1)!}\,k^{-1}(d-1)^{(d-1)/k}.
\end{equation}

\section{Proof of Theorem \ref{thm:main}}
\label{sec:proof-main}

We now combine the kernel reduction with the product-layer estimates.  We
recall the notation used below.  The constants $B_0$ and $B_R$ denote
the numbers of full $k$-uniform $(t+2)$-stars supplied by
Proposition~\ref{prop:cover-kernel}.  The constant $Z$ is the coefficient
from \eqref{eq:Z-bound}, while $P$ is defined in \eqref{eq:P}.  Finally,
$H_0$ is the coefficient of the Hilton--Milner pure product layer in
\eqref{eq:H0A0}, and $\Delta_S,\Delta_T$ are the coefficient losses
defined in \eqref{eq:DeltaS} and \eqref{eq:DeltaT}.

The proof below will require several numerical inequalities which allow the
error terms from the kernel reduction to be absorbed into the product-layer
gaps.  We record these consequences of the assumption
$n\ge 100\cdot 2^k k^7$ in the following lemma.

\begin{lemma}
\label{lem:absorption}
For $k\ge3$, $1\le t\le k-2$, and $n\ge100\cdot 2^k k^7$, the following inequalities hold:
\begin{align}
n&\ge 2(k+1),\quad
n\ge 2L(k+1),\quad
n\ge 2L(r+1),
\label{eq:absorption-loss}\\
n&\ge
\frac{2H_0L(k+1)}{\Delta_S},
\label{eq:absorption-star-loss}\\
n^\theta&\ge \frac{4B_0Z}{H_0},\quad
n^\theta\ge \frac{4\bigl(B_R+\binom q2\bigr)Z}{\Delta_S},\quad
n^\theta\ge\frac{4(B_R+1)Z}{\Delta_T}.
\label{eq:absorption-error}
\end{align}
\end{lemma}

\begin{proof}
We first collect some elementary estimates.  From the definitions of $Z$
and $P$, we get
\[
 \frac ZP\le d\le k.
\]
Since $H_0=Pq^\theta$ and $q\le k$, we have $H_0\le Pk$. On the other hand, $q\ge1$, so $H_0\ge P$.

By the mean-value theorem, we get
\[
 q^\theta-(q-1)^\theta
 =\theta \xi^{\theta-1}
\]
for some $\xi\in(q-1,q)$.  Since $\theta-1=-1/k$, this gives $q^\theta-(q-1)^\theta
 \ge \theta q^{-1/k}$. Hence, we obtain
\[
 \Delta_S
 =P\bigl(q^\theta-(q-1)^\theta\bigr)
 \ge P\theta q^{-1/k}
 \ge \frac Pk,
\]
where the last inequality uses $q\le k$ and $k\ge3$.

Similarly, using the expression for $\Delta_T$ in \eqref{eq:DeltaT}, we get
\[
 \frac{\Delta_T}{P}
 =r\left[
 \left(1+\frac1r\right)^{1-r/k}
 -\left(1-\frac1r\right)^{(r-1)/k}
 \right].
\]
Since $0\le 1-r/k\le1$, we have
\[
 \left(1+\frac1r\right)^{1-r/k}\ge1.
\]
Also, since $0\le (r-1)/k\le1$, Bernoulli's inequality
and concavity give
\[
 \left(1-\frac1r\right)^{(r-1)/k}
 \le 1-\frac{r-1}{kr}.
\]
Therefore,
\[
 \frac{\Delta_T}{P}
 \ge r\left[1-\left(1-\frac{r-1}{kr}\right)\right]
 =\frac{r-1}{k}\ge\frac1k.
\]

The covering numbers from Proposition~\ref{prop:cover-kernel} satisfy
\[
 B_0=\binom kt k^2\le2^k k^2,
 \qquad
 B_R=rk^2\le k^3,
 \qquad
 \binom q2\le\frac{k^2}{2}.
\]
Moreover, since $L(u)=2d\theta u+d(d-1)$ and $d\le k$, $\theta\le1$, we have $L(u)\le 2ku+k^2$. Thus, as $r+1\le k+1$, we get
\[
 L(k+1)\le4k^2,
 \qquad
 L(r+1)\le4k^2.
\]

Put $n_0=100\cdot 2^k k^7$. Since $\theta=(k-1)/k\ge2/3$, we have $n_0^\theta
 =100^\theta 2^{k-1}k^{7-7/k}
 \ge8\cdot 2^k k^4$.
The estimates above imply
\[
 \frac{4B_0Z}{H_0}
 \le 4\cdot 2^k k^2\cdot k
 =4\cdot 2^k k^3,
\]
and
\[
 \frac{2H_0L(k+1)}{\Delta_S}
 \le
 2\cdot k\cdot 4k^2\cdot k
 =8k^4.
\]
Similarly,
\[
 \frac{4\bigl(B_R+\binom q2\bigr)Z}{\Delta_S}
 \le
 4\left(k^3+\frac{k^2}{2}\right)k^2
 \le 8k^5,
\]
and
\[
 \frac{4(B_R+1)Z}{\Delta_T}
 \le
 4(k^3+1)k^2
 \le 8k^5.
\]
Since $2^k\ge k$ for $k\ge3$, we have
\[
 n_0^\theta\ge8\,2^k k^4\ge8k^5.
\]
Also, since $n_0\ge8k^4$, it is clear that
\[
 n_0\ge 2(k+1),\qquad
 n_0\ge 8k^2\ge 2L(k+1),\qquad
 n_0\ge 8k^2\ge 2L(r+1).
\]
Therefore, for every $n\ge n_0$, all inequalities
\eqref{eq:absorption-loss}--\eqref{eq:absorption-error} hold.
\end{proof}

\begin{proof}[\bf Proof of Theorem \ref{thm:main}.]
Since $\lambda=(k-1)!\rho$, it suffices to compare the
$\lambda$-spectral radii.  Extend $\Fcal$ to a maximal nontrivial
$t$-intersecting family $\Gcal$.  Monotonicity from
Lemma~\ref{lem:basic} gives
\[
 \lambda(\Fcal)\le\lambda(\Gcal).
\]

We distinguish two cases according to the $t$-covering number of
$\Gcal$. Suppose first that $\tau_t(\Gcal)\ge t+2$.  By
Proposition~\ref{prop:cover-kernel}, the family $\Gcal$ is covered by
$B_0$ complete $(t+2)$-stars.  Hence Lemma~\ref{lem:basic} and
\eqref{eq:Z-bound} imply
\[
 \lambda(\Gcal)
 \le B_0Z\,n^{(d-1)\theta}.
\]
On the other hand, we may lower-bound
$\lambda(\Hcal_{n,k,t})$ by retaining only its pure product layer.
The corresponding kernel is the full star with $q$ edges.  Since
$d\le k$ and $n\ge2(k+1)$, Lemmas~\ref{lem:product} and
\ref{lem:beta-lower}, applied with $u=k+1$, give
\[
 \lambda(\Hcal_{n,k,t})
 \ge
 H_0\left(1-\frac{L(k+1)}n\right)n^{d\theta}.
\]
By \eqref{eq:absorption-loss}, we get $1-\frac{L(k+1)}n\ge\frac12$, and therefore $\lambda(\Hcal_{n,k,t})
 \ge \frac{H_0}{2}n^{d\theta}$. Moreover, by the first inequality in \eqref{eq:absorption-error}, we get
\[
 B_0Z\,n^{(d-1)\theta}
 \le \frac{H_0}{4}n^{d\theta}.
\]
Thus, we obtain
\[
 \lambda(\Gcal)
 \le \frac{H_0}{4}n^{d\theta}
 < \frac{H_0}{2}n^{d\theta}
 \le \lambda(\Hcal_{n,k,t}).
\]
Hence, this case cannot give an extremal family.

It remains to consider the case $\tau_t(\Gcal)=t+1$.  Let $\Kcal$
be the kernel of $\Gcal$, as defined in
Section~\ref{sec:cover-kernel}.  By
Proposition~\ref{prop:cover-kernel}, either
\[
 \Kcal\subseteq \{X\cup\{y\}:y\in Y\}
 \quad\text{or}\quad
 \Kcal\subseteq \binom{Z}{r}.
\]
If $\Kcal=\{X\cup\{y\}:y\in Y\}$, then
$\Gcal=\Hcal_{n,k,t}$; if $\Kcal=\binom{Z}{r}$, then
$\Gcal=\Acal_{n,k,t}$.  Thus it remains only to show that no proper
subkernel can be extremal.

First suppose that $\Kcal$ is a proper subkernel of the full star with
$q$ edges.  Write the full star as
\[
 \{X\cup\{y\}:y\in Y\},
 \qquad |X|=t,\qquad |Y|=q,
\]
and take $U=X\cup Y$ as the kernel support in Lemma~\ref{lem:product}.  Then $|U|=k+1$, so
the pure product layers of $\Gcal$ and $\Hcal_{n,k,t}$ use the same
exterior vertex set $[n]\setminus U$.

Since $\Kcal$ is proper, it has at most $q-1$ edges.  By the
coefficient loss \eqref{eq:DeltaS} and the product formula, the pure
product layer generated by $\Kcal$ has spectral radius at most
\[
 (H_0-\Delta_S)\beta_d(n-k-1)
 \le (H_0-\Delta_S)n^{d\theta},
\]
where the last inequality follows from \eqref{eq:beta-upper}.

It remains to bound the edges outside this pure layer.  Every edge of
$\Lcal_n(\Kcal)$ that is not in the pure layer contains at least two
vertices of $Y$.  Hence these edges are contained in the union of the
$\binom q2$ complete $(t+2)$-stars with centers
\[
 X\cup\{y,y'\},
 \qquad \{y,y'\}\in\binom Y2.
\]
Moreover, by Proposition~\ref{prop:cover-kernel}, 
$\Gcal\setminus\Lcal_n(\Kcal)$ is covered by at most $B_R$ complete
$(t+2)$-stars.  Therefore, we get
\[
 \lambda(\Gcal)
 \le
 (H_0-\Delta_S)n^{d\theta}
 +\left(B_R+\binom q2\right)Z\,n^{(d-1)\theta}.
\]

On the other hand, the full pure product layer of $\Hcal_{n,k,t}$
gives
\[
 \lambda(\Hcal_{n,k,t})
 \ge
 \left(H_0-\frac{H_0L(k+1)}n\right)n^{d\theta}.
\]
By \eqref{eq:absorption-star-loss}, we get
\[
 \frac{H_0L(k+1)}n\le\frac{\Delta_S}{2},
\]
and by \eqref{eq:absorption-error}, we get
\[
 \left(B_R+\binom q2\right)Z\,n^{(d-1)\theta}
 \le\frac{\Delta_S}{4}n^{d\theta}.
\]
Thus, we obtain
\[
 \lambda(\Gcal)
 \le\left(H_0-\frac{3\Delta_S}{4}\right)n^{d\theta}
 <
 \left(H_0-\frac{\Delta_S}{2}\right)n^{d\theta}
 \le\lambda(\Hcal_{n,k,t}).
\]
Hence no proper subkernel of the full star with $q$ edges can be
extremal.

Finally, suppose that $\Kcal$ is a proper subgraph of the complete
$r$-graph $\Tcal^{(r)}$ on an $(r+1)$-element set $Z$. Recall that 
$\Mcal_{n,k}(\Kcal)$ is the pure product layer generated by $\Kcal$.
Since $\Kcal$ is proper, at least one edge is missing. Hence the loss
estimate gives
\[
 \lambda(\Mcal_{n,k}(\Kcal))
 \le
 \lambda\!\left(\Mcal_{n,k}(\Tcal^{(r)})\right)
 -\frac{\Delta_T}{2}n^{d\theta}.
\]
On the other hand, by Proposition~\ref{prop:cover-kernel}, the edges of $\Gcal$ outside $\Mcal_{n,k}(\Kcal)$
contribute at most
\[
 (B_R+1)Z\,n^{(d-1)\theta}
 \le \frac{\Delta_T}{4}n^{d\theta}.
\]
Therefore, by subadditivity and monotonicity of $\lambda$, we get
\[
\begin{aligned}
 \lambda(\Gcal)
 &\le \lambda(\Mcal_{n,k}(\Kcal))
 +(B_R+1)Z\,n^{(d-1)\theta} \\
 &\le
 \lambda\!\left(\Mcal_{n,k}\!\left(\Tcal^{(r)}\right)\right)
 -\frac{\Delta_T}{2}n^{d\theta}
 +\frac{\Delta_T}{4}n^{d\theta} \\
 &<
 \lambda\!\left(\Mcal_{n,k}\!\left(\Tcal^{(r)}\right)\right)\le
 \lambda(\Acal_{n,k,t}).
\end{aligned}
\]
Thus this case cannot give an extremal family.

It remains only to justify the equality case. The two-section of
$\Acal_{n,k,t}$ is connected. Indeed, every vertex outside $[t+2]$
appears in an edge together with a $(t+1)$-subset of $[t+2]$, while
the core $[t+2]$ is itself connected. Similarly, the two-section of
$\Hcal_{n,k,t}$ is connected: the core $[k+1]$ is connected through
the pure product layer, and every exterior vertex appears in a pure-layer
edge meeting this core. Hence, by the strict monotonicity in
Lemma~\ref{lem:strict-monotonicity}, every proper subfamily of either
candidate has strictly smaller spectral radius. Therefore equality holds
only for the corresponding full candidate family.

We finish the proof by comparing the two candidates asymptotically. The
pure product layer of $\Hcal_{n,k,t}$ has kernel equal to the full star
with $q$ edges, and all remaining edges contribute only
$O_{k,t}(n^{d-1})$ edges. Similarly, the pure product layer of
$\Acal_{n,k,t}$ has kernel equal to the complete $r$-graph on an
$(r+1)$-element set, and its remaining edges also number
$O_{k,t}(n^{d-1})$. Therefore, by Lemmas~\ref{lem:edge-bound} and
\ref{lem:beta-lower},
\[
 \lambda(\Hcal_{n,k,t})
 =
 \bigl(H_0+o(1)\bigr)n^{d\theta},
 \qquad
 \lambda(\Acal_{n,k,t})
 =
 \bigl(A_0+o(1)\bigr)n^{d\theta},
\]
as $n\to\infty$, with $k,t$ fixed.

Using \eqref{eq:H0A0}, and recalling that
$r=t+1$, $q=k-t+1$, and $d=k-r=k-t-1$, we get
\[
 \left(\frac{A_0}{H_0}\right)^k
 =
 \frac{r^r(r+1)^d}{q^{k-1}}
 =
 \frac{(t+1)^{t+1}(t+2)^{k-t-1}}
 {(k-t+1)^{k-1}}.
\]
Thus, we get $\frac{A_0}{H_0}=\exp\!\left(\frac{\Phi_t(k)}{k}\right)$, where
\begin{equation}
\label{eq:Phi}
 \Phi_t(x)=
 (x-t-1)\log(t+2)
 +(t+1)\log(t+1)
 -(x-1)\log(x-t+1),
 \qquad x>t+1.
\end{equation}
Consequently, we have
\begin{equation}\label{eq:ratio-asymptotic}
 \frac{\lambda(\Acal_{n,k,t})}
 {\lambda(\Hcal_{n,k,t})}
 =
 \exp\!\left(\frac{\Phi_t(k)}{k}\right)+o(1)
 \qquad(n\to\infty).
\end{equation}

It remains to determine the sign of $\Phi_t(k)$. We first show that
$\Phi_t$ is strictly decreasing on $[t+2,\infty)$. Put
$y=x-t+1$. Then $y\ge3$, and
\[
 \Phi_t'(x)
 =
 \log\frac{t+2}{y}
 -1-\frac{t-2}{y}
 \eqqcolon g_t(y).
\]
Moreover, we get
\[
 g_t'(y)
 =
 \frac{t-2-y}{y^2}.
\]
If $t\le5$, then $t-2\le3$, so $g_t$ is decreasing on
$[3,\infty)$. Hence
\[
 g_t(y)\le g_t(3)
 =
 \log\frac{t+2}{3}-\frac{t+1}{3}.
\]
By $\log z\le z-1$, we get
\[
 g_t(3)
 \le
 \frac{t-1}{3}-\frac{t+1}{3}
 =
 -\frac23<0.
\]
Thus $g_t(y)<0$ for all $y\ge3$.

If $t\ge6$, then $g_t$ increases on $[3,t-2]$ and decreases on
$[t-2,\infty)$. Therefore its maximum occurs at $y=t-2$, and
\[
 g_t(y)\le g_t(t-2)
 =
 \log\frac{t+2}{t-2}-2
 \le
 \log 2-2<0.
\]
Thus, in all cases, we get $
 \Phi_t'(x)<0$ for every $x\ge t+2$.

Next,
\[
 \Phi_t(t+2)
 =
 \log(t+2)+(t+1)\log\frac{t+1}{3}.
\]
This is positive for all $t\ge1$. In fact, for $t=1$ it equals
$\log(4/3)>0$, for $t=2$ it equals $\log4>0$, and for
$t\ge3$ both terms are nonnegative and the first is positive. Also,
\[
 \Phi_t(x)\to -\infty
 \qquad\text{as }x\to\infty,
\]
since the negative term $-(x-1)\log(x-t+1)$ dominates the linear term
$(x-t-1)\log(t+2)$. Hence $\Phi_t$ has a unique zero, denoted
$\kappa_t$. Equivalently, $\kappa_t$ is the unique solution of
\eqref{eq:transition-equation}.

By \eqref{eq:ratio-asymptotic}, $\Acal_{n,k,t}$ asymptotically has
larger spectral radius than $\Hcal_{n,k,t}$ when
$\Phi_t(k)>0$, that is, when $t+2\le k<\kappa_t$. Conversely,
$\Hcal_{n,k,t}$ asymptotically has larger spectral radius when
$\Phi_t(k)<0$, that is, when $k>\kappa_t$.

Finally, a $p$-adic valuation comparison shows that
$\kappa_t\notin\mathbb Z$ for every $t$; see
Appendix~\ref{appendix} for the details. Therefore no integer value of
$k$ gives an asymptotic tie between the two candidates. This completes
the proof.
\end{proof}

\begin{remark}
Table~\ref{tab:transition} gives numerical approximations to the first few
transition points $\kappa_t$. Since $\kappa_t\notin\mathbb Z$, there is
no tie at an integer value of $k$: the Frankl family has larger asymptotic
spectral radius for $t+2\le k\le \lfloor \kappa_t\rfloor$, while the
Hilton--Milner family has larger asymptotic spectral radius for
$k\ge \lceil \kappa_t\rceil$.
\end{remark}

\begin{table}[htbp]
\centering
\caption{The first transition values $\kappa_t$.}
\label{tab:transition}
\begin{tabular}{c@{\qquad}c@{\qquad}c@{\qquad}c@{\qquad}c@{\qquad}c}
\hline
$t$ & $\kappa_t$ & & $t$ & $\kappa_t$\\
\hline
1 & 3.384653 & &8 & 17.837590\\
2 & 5.493935 & &9 & 19.875060\\
3 & 7.575890 & &10 & 21.909618\\
4 & 9.643184 & &11 & 23.941705\\
5 & 11.700799 & &12 & 25.971667\\
6 & 13.751401 & &13 & 27.999777\\
7 & 15.796631 & &14 & 30.026261\\
\hline
\end{tabular}
\end{table}

We also record the asymptotic location of the transition point. Let
$x=2t+c$. For $c=O(\log t)$, Taylor expansion in the definition
\eqref{eq:Phi} gives, uniformly in this range,
\[
 \Phi_t(2t+c)
 =
 \log t+1-2c
 +O\!\left(\frac{c^2+1}{t}\right).
\]
In particular, taking $c=0$ and $c=\log t$ shows that
$\Phi_t(2t)>0$ and $\Phi_t(2t+\log t)<0$ for all sufficiently large
$t$. Since $\Phi_t$ is strictly decreasing on $[t+2,\infty)$, this
implies $\kappa_t=2t+O(\log t)$.

Now put $c_0=\frac12\log t+\frac12$. The same expansion gives
\[
 \Phi_t(2t+c_0)
 =
 O\!\left(\frac{(\log t)^2}{t}\right).
\]
Moreover, from
\[
 \Phi_t'(x)
 =
 \log\frac{t+2}{x-t+1}
 -1-\frac{t-2}{x-t+1},
\]
we have, uniformly for $x=2t+O(\log t)$,
\[
 \Phi_t'(x)=-2+O\!\left(\frac{\log t}{t}\right).
\]
Thus $\Phi_t'$ is bounded away from $0$ in a neighborhood of
$2t+c_0$. By the mean-value theorem, we get
\[
 \kappa_t-(2t+c_0)
 =
 O\!\left(\frac{(\log t)^2}{t}\right).
\]
Therefore, we obtain
\begin{equation}\label{eq:kappa-asymptotic}
 \kappa_t
 =
 2t+\frac12\log t+\frac12
 +O\!\left(\frac{(\log t)^2}{t}\right).
\end{equation}

We conclude this section with a finite-$n$ comparison of the two
candidate families. For the Hilton--Milner family and the Frankl family,
respectively, the numbers of vertices outside the kernel supports are
\[
 N_H=n-(q+t)=n-k-1,\qquad N_A=n-(r+1)=n-t-2.
\]

The automorphism group of $\Hcal_{n,k,t}$ has three vertex classes:
$X$, $Y$, and $[n]\setminus(X\cup Y)$. By Lemma
\ref{lem:symmetrization}, a maximizing vector may be taken to be constant
on these classes. Write the corresponding coordinates as $a,b,c$.
Counting an edge by the number $j$ of its vertices in $Y$ gives
\begin{align}
 \rho(\Hcal_{n,k,t})
 =k\max\Bigg\{&
 a^t\sum_{j=1}^{q-1}
 \binom qj\binom{N_H}{q-1-j}b^jc^{q-1-j}
 +t a^{t-1}b^q:\notag\\[-1mm]
 &a,b,c\ge0,\quad ta^k+qb^k+N_Hc^k=1\Bigg\}.
 \label{eq:LambdaH}
\end{align}

Similarly, the automorphism group of $\Acal_{n,k,t}$ has two vertex
classes: $Z$ and its complement. An edge contains either exactly
$r=t+1$ vertices of $Z$, or all $r+1=t+2$ vertices of $Z$.
Thus
\begin{align}
 \rho(\Acal_{n,k,t})
 =k\max\Bigg\{&
 (t+2)\binom{N_A}{d}a^{t+1}c^d
 +\binom{N_A}{d-1}a^{t+2}c^{d-1}:\notag\\[-1mm]
 &a,c\ge0,\quad (t+2)a^k+N_Ac^k=1\Bigg\}.
 \label{eq:LambdaA}
\end{align}
Hence, for fixed $t$ and $k$, \eqref{eq:LambdaH} and
\eqref{eq:LambdaA} reduce the exact comparison to a low-dimensional
optimization problem.

For a simpler sufficient comparison, set
\begin{equation}\label{eq:pure-values}
 M_H=H_0\beta_d(n-k-1),\qquad
 M_A=A_0\beta_d(n-t-2),\qquad
 B_S=\binom q2.
\end{equation}

\begin{prop}
\label{prop:certificates}
Let \(M_H\), \(M_A\), and \(B_S\) be defined by \eqref{eq:pure-values}.
If $M_H>M_A+Z n^{(d-1)\theta}$, then $\rho(\Hcal_{n,k,t})>\rho(\Acal_{n,k,t})$; 
if $M_A>M_H+B_SZ n^{(d-1)\theta}+(k!t)^\theta$, then $\rho(\Acal_{n,k,t})>\rho(\Hcal_{n,k,t})$.
\end{prop}

\begin{proof}
Since $\lambda=(k-1)!\rho$, it suffices to prove the corresponding
strict comparisons for $\lambda$. The quantities $M_H$, $M_A$, and
the error terms below are all on the $\lambda$-scale.

Each family contains its pure product layer, and therefore
\[
 \lambda(\Hcal_{n,k,t})\ge M_H,
 \qquad
 \lambda(\Acal_{n,k,t})\ge M_A.
\]

We first upper-bound the part of $\Acal_{n,k,t}$ outside its pure layer.
The only such edges contain all $t+2$ vertices of $Z$; they form one
complete $(t+2)$-star. By Lemma \ref{lem:basic} and
\eqref{eq:Z-bound}, their total contribution is at most
$Zn^{(d-1)\theta}$. Hence, $\lambda(\Acal_{n,k,t})
 \le M_A+Zn^{(d-1)\theta}$. Thus, if $M_H>M_A+Zn^{(d-1)\theta}$,
then
\[
 \lambda(\Hcal_{n,k,t})
 \ge M_H
 > M_A+Zn^{(d-1)\theta}
 \ge \lambda(\Acal_{n,k,t}),
\]
and so $\rho(\Hcal_{n,k,t})>\rho(\Acal_{n,k,t})$.

Next we upper-bound the part of $\Hcal_{n,k,t}$ outside its pure layer.
Every remaining lift edge contains at least two vertices of $Y$. Choosing
such a pair shows that these edges are covered by $B_S=\binom q2$ complete $(t+2)$-stars. Again by Lemma \ref{lem:basic} and
\eqref{eq:Z-bound}, their total contribution is at most $B_SZn^{(d-1)\theta}$. The remaining $t$ exceptional edges contribute at most $(k!t)^\theta$ 
by Lemma \ref{lem:edge-bound}. Therefore
\[
 \lambda(\Hcal_{n,k,t})
 \le M_H+B_SZn^{(d-1)\theta}+(k!t)^\theta.
\]
Hence, if $M_A>M_H+B_SZn^{(d-1)\theta}+(k!t)^\theta$, then
\[
 \lambda(\Acal_{n,k,t})
 \ge M_A
 > M_H+B_SZn^{(d-1)\theta}+(k!t)^\theta
 \ge \lambda(\Hcal_{n,k,t}),
\]
and consequently $\rho(\Acal_{n,k,t})>\rho(\Hcal_{n,k,t})$.
This proves both sufficient conditions.
\end{proof}

\section{Concluding remarks}

The explicit threshold \(100\cdot 2^k k^7\) is not meant to be optimal. The main
losses in the proof come from summing the spectral radii of overlapping
complete stars and from replacing exact product factors by uniform linear
bounds. A more refined argument that takes the Perron vector into account,
possibly combined with shifting or branching methods as in recent work on
the spectral Erd\H{o}s matching problem \cite{KangLuYuanZhou}, could reduce
the required ground-set size.

The phase equation \eqref{eq:transition-equation} reveals a phenomenon absent
from the ordinary cardinality comparison. Asymptotically, the first Frankl
family remains spectrally larger than the Hilton--Milner family for a
logarithmically wide range beyond \(k=2t+1\). It would be interesting to know
whether the same transition rule holds for finite \(n\) throughout the range
of Theorem \ref{thm:main}; that is, whether the sign of \(\Phi_t(k)\) always
determines the exact comparison between \(\Acal_{n,k,t}\) and
\(\Hcal_{n,k,t}\). Proposition \ref{prop:certificates} reduces this to an
explicit finite-\(n\) comparison near the phase boundary.

\appendix

\section{Non-integrality of the solution of \eqref{eq:transition-equation}}
\label{appendix}

In this appendix we prove that the solution \(\kappa_t\) of
\eqref{eq:transition-equation} is not an integer.

Suppose, for a contradiction, that \(x=\kappa_t\in \mathbb Z\), and set $y=x-t+1$. Then \eqref{eq:transition-equation} becomes
\begin{equation}
\label{eq:prime-comparison}
  (t+2)^{y-2}(t+1)^{t+1}=y^{y+t-2},
  \qquad y>2 .
\end{equation}
Since \(t+1\) and \(t+2\) are coprime and both are greater than \(1\), the
left-hand side of \eqref{eq:prime-comparison} has at least two distinct prime
divisors. Hence \(y\) has at least two distinct prime divisors. In particular,
\(y\) is composite, and so \(y\ge 4\).

Let \(p\mid t+2\). Since \(p\nmid t+1\), comparing \(p\)-adic valuations in
\eqref{eq:prime-comparison} gives
\[
  (y-2)v_p(t+2)=(y+t-2)v_p(y).
\]
Equivalently,
\[
  v_p(y)=\frac{y-2}{y+t-2}\,v_p(t+2).
\]
Write
\[
  \frac{y-2}{y+t-2}=\frac{a}{b},
  \qquad \gcd(a,b)=1.
\]
Since \(t\ge 1\), we have \(0<a/b<1\), and hence \(b>1\). The equality above
implies that
\[
  v_p(y)=\frac{a}{b}v_p(t+2).
\]
Thus \(b\mid v_p(t+2)\) for every prime \(p\mid t+2\). Therefore all prime
exponents in the factorization of \(t+2\) are divisible by \(b>1\), so
\(t+2\) is a nontrivial perfect power.

The same argument applies to \(t+1\). Indeed, if \(p\mid t+1\), then
\(p\nmid t+2\), and \eqref{eq:prime-comparison} gives
\[
  (t+1)v_p(t+1)=(y+t-2)v_p(y).
\]
Thus, we have
\[
  v_p(y)=\frac{t+1}{y+t-2}\,v_p(t+1).
\]
Since \(y\ge 4\), we have
\[
  0<\frac{t+1}{y+t-2}<1.
\]
Writing this fraction in lowest terms, its denominator is greater than \(1\),
and the same divisibility argument shows that every prime exponent in
\(t+1\) is divisible by this denominator. Hence \(t+1\) is also a nontrivial
perfect power.

Thus the consecutive integers \(t+1\) and \(t+2\) are both nontrivial perfect
powers. By Mih\v{a}ilescu's theorem \cite{Mihailescu2004}, the only such pair is
\(8\) and \(9\). Therefore,  $t+1=8$ and $t+2=9$, so \(t=7\). Substituting this into \eqref{eq:prime-comparison} yields
$9^{y-2}8^8=y^{y+5}$, or equivalently $3^{2(y-2)}2^{24}=y^{y+5}$. Comparing \(2\)-adic valuations gives $24=(y+5)v_2(y)$. In particular, \(y\) is even and \(y+5\mid 24\). Since \(y>2\), this forces $y\in\{7,19\}$. Both values are prime, contradicting the fact that \(y\) is composite. This
contradiction proves that \(\kappa_t\notin\mathbb Z\).

\section*{Declaration of competing interests}
The authors declare no competing interests.

\section*{Data availability}
No data were used for the research described in this article.

\section*{Declaration of AI usage}
The authors acknowledge the use of AI tools during the exploratory stage of this project. All
mathematical arguments and proofs in the final manuscript were checked and written by the
authors.

\section*{Acknowledgements}
Lihua Feng is supported by
 National Natural Science Foundation of China (Nos. 12271527 and 12471022). Lu Lu is supported by National Natural Science Foundation of China (No. 12371362).
The authors contributed equally to this work.

\end{document}